\documentclass[11pt,a4paper]{article}

\usepackage[utf8]{inputenc}
\usepackage[english]{babel}
\usepackage[colorinlistoftodos]{todonotes}
\usepackage{amsmath, amssymb, amsthm}
\usepackage{mathtools} 
\usepackage{mathrsfs}  

\usepackage{geometry}
\usepackage{microtype} 
\usepackage{enumitem}  
\usepackage{booktabs}  
\usepackage{xcolor}    

\usepackage{tikz}
\usepackage{tikz-cd}
\usetikzlibrary{patterns}

\usepackage{hyperref}
\hypersetup{
    colorlinks=true,
    linkcolor=blue,
    citecolor=red,
    filecolor=magenta,
    urlcolor=cyan,
    pdftitle={On the completeness of generalized hierarchical spline spaces},
    pdfauthor={Author Name}
}
\usepackage[nameinlink]{cleveref} 

\newcommand{\R}{\mathbb{R}}

\newcommand{\dd}{\,\mathrm{d}}       
\DeclareMathOperator{\Span}{span}    
\DeclareMathOperator{\supp}{supp}
\DeclareMathOperator{\Ker}{ker}

\newcommand{\Ccell}{\mathrm{G}}
\newcommand{\Mcell}{\mathrm{M}}
\newcommand{\Fface}{\mathrm{F}}
\newcommand{\Ggrid}{\mathcal{G}}
\newcommand{\Bbasis}{\mathcal{B}}
\newcommand{\Dbasis}{\mathcal{D}}
\newcommand{\Hbasis}{\mathcal{H}}
\newcommand{\Hspace}{\mathbb{H}}
\newcommand{\Sspace}{\mathbb{S}}
\newcommand{\Uspace}{\mathbb{U}}
\newcommand{\PWspace}{\mathbb{PW}}
\newcommand{\Ggraph}{\mathfrak{G}}

\newcommand{\dec}{\mathrm{dec}}
\newcommand{\ect}{\mathrm{T}}
\newcommand{\h}{\mathrm{H}}
\newcommand{\im}{\mathrm{i}}
\newcommand{\vd}{{\boldsymbol{d}}}

\newcommand\blfootnote[1]{%
  \begingroup
  \renewcommand\thefootnote{}\footnote{#1}%
  \addtocounter{footnote}{-1}%
  \endgroup
}

\theoremstyle{plain}
\newtheorem{theorem}{Theorem}[section]
\newtheorem{lemma}[theorem]{Lemma}

\newtheorem{corollary}[theorem]{Corollary}

\theoremstyle{definition}
\newtheorem{definition}[theorem]{Definition}
\newtheorem{example}[theorem]{Example}
\newtheorem{assumption}[theorem]{Assumption}

\theoremstyle{remark}
\newtheorem{remark}[theorem]{Remark}

\makeatletter
\newcommand{\subjclass}[2][2020]{%
  \let\@oldtitle\@title%
  \gdef\@title{\@oldtitle\footnotetext{\emph{#1 Mathematics Subject Classification.} #2}}%
}
\makeatother

\title{\textbf{On the completeness of generalized hierarchical spline spaces}}

\author{
  \text{Ahmed Oufqir}\textsuperscript{1},
  \text{Carla Manni}\textsuperscript{2},
  \text{Hendrik Speleers}\textsuperscript{2} \\[1em]
  \small \textsuperscript{1}The UM6P Vanguard Center, Mohammed VI Polytechnic University, Rabat, Morocco \\
  \small \textsuperscript{2}Department of Mathematics, University of Rome Tor Vergata, Rome, Italy   
}

\date{}

\begin{document}

\maketitle
\blfootnote{
    \textbf{E-mail addresses:} 
    \texttt{ahmed.oufqir@um6p.ma} (Ahmed Oufqir),
    \texttt{manni@mat.uniroma2.it} (Carla Manni),
    \texttt{speleers@mat.uniroma2.it} (Hendrik Speleers).
}

\begin{abstract}
    We introduce a general theoretical approach to hierarchical spline spaces that replaces the classical constructive definition --- based on basis selection --- with a descriptive formulation in terms of regularity constraints. Specifically, we define generalized hierarchical spline spaces on multi-level domains as collections of piecewise functions satisfying \emph{hierarchical contact} conditions across interfaces between refinement levels. The proposed framework applies to a broad class of local function spaces and relies on a minimal abstract requirement, the \emph{extension assumption}, rather than on specific polynomial properties. Within this framework, we identify rules under which the hierarchical selection mechanism yields a complete basis, in the sense that it spans exactly the space characterized by the contact conditions. As an application, we consider Tchebycheffian spline spaces. We show that spaces generated by extended complete Tchebycheff (ECT) systems fit in this framework, thereby establishing the completeness of hierarchical Tchebycheffian splines. This demonstrates that the proposed theory naturally extends beyond the polynomial setting and provides a unified foundation for hierarchical constructions in more general spline spaces.
\end{abstract}

\section{Introduction}

Adaptive refinement strategies are fundamental to  numerical analysis and geometric modeling, particularly within the framework of isogeometric analysis \cite{cottrel_isogeometric_2009}, a well-established  paradigm for the discretization of differential problems. The construction of local refinement schemes has largely relied on the concept of hierarchical spline structures, such as hierarchical B-splines (HB-splines) and truncated hierarchical B-splines (THB-splines) \cite{giannelli_strongly_2014}. A comprehensive overview of these foundations and their recent developments can be found in \cite{buffa_mathematical_2022,lyche_foundations_2018}.

Classically, these hierarchical spaces are defined through a \emph{constructive} approach. This method involves a selection mechanism, where specific basis functions are chosen from a sequence of nested spline spaces defined on a hierarchy of grids. While effective for implementation, this perspective obscures the intrinsic nature of the function space. It treats the space as a consequence of the basis selection rather than an independent mathematical object defined by regularity constraints. In this paper, we propose a shift toward a \emph{descriptive} formulation. Inspired by the polynomial framework presented in \cite{mokris_completeness_2014}, we define the generalized hierarchical spline space directly on a multi-level domain by imposing regularity constraints, termed \emph{hierarchical contact}, across the interfaces of the mesh cells. This allows us to characterize the space as the collection of all piecewise functions satisfying these regularity conditions, independent of any specific basis representation.

The central contribution of this work is the extension of this hierarchical framework beyond the polynomial setting. While polynomial splines \cite{lyche_foundations_2018,schumaker_spline_2007} differ in their specific algebraic properties from other spaces, the geometric principles of hierarchy are universal. We develop a theory for \emph{generalized hierarchical spline spaces} that relies on a minimal set of abstract assumptions. Central to this generalization is the \emph{extension assumption} (Assumption~\ref{ass:extension}), which postulates that local sections of the spline space can be extended globally while preserving regularity. This abstraction allows our results to cover a broad class of spaces.

Within this generalized framework, we address the fundamental question of \emph{completeness}. The constructive definition (via basis selection) and the descriptive definition (via hierarchical contact) must be shown to yield the same space. We identify rules under which the hierarchical basis, constructed via a selection mechanism analogous to that in \cite{lyche_foundations_2018,mokris_completeness_2014}, spans exactly the space defined by contact conditions.
This result generalizes the completeness results for polynomial tensor-product splines established in \cite{mokris_completeness_2014} to the abstract setting; see Theorem~\ref{thm:hier_complete}.

This abstract investigation is motivated by the fact that there is a well-established theory for spline spaces incorporating non-polynomial functions. A relevant class is given by Tchebycheffian splines: smooth piecewise functions whose pieces are drawn from (possibly different) extended complete Tchebycheff (ECT) spaces, which generalize algebraic polynomial spaces by including trigonometric, exponential, and hyperbolic functions, among others. Under suitable assumptions, Tchebycheffian splines possess a basis with many of the desirable properties of classical B-splines, while they extend the geometric expressive power of the underlying function space; see \cite{lyche_tchebycheffian_2019, schumaker_spline_2007} and references therein.

Tchebycheffian splines have been effectively applied in different contexts, including data approximation, signal processing, computer aided design, and isogeometric analysis. In particular, they allow for an exact representation of a broad class of geometries (e.g., conic sections, helices, and cycloids), which naturally arise in geometric modeling but cannot be exactly captured by polynomial splines. Moreover, for the relevant class of ECT-spaces  given by kernels of linear differential operators with constant coefficients, Tchebycheffian (B-)splines maintain the same differential structure of the polynomial case and are equipped with efficient algorithms for their evaluation and manipulation \cite{speleers_algorithm_2022}. The latter Tchebycheffian (B-)splines have been utilized as an effective problem-dependent substitute for standard polynomial B-splines and their rational extension (NURBS) in isogeometric methods; see \cite{manni_survey_2017,raval_tchebycheffian_2023} and references therein. They have also been extended to local tensor-product structures in \cite{manni_local_2014,raval_adaptive_2024}.

Our general theory applies to Tchebycheffian spline spaces, as a special but important instance.
We show that Tchebycheffian spline spaces satisfy our abstract assumptions. To this end, the existence of compatible weight systems \cite{buchwald_bsplines_2003,lyche_tchebycheffian_2019,mazure_finding_2011} and the geometric constraints related to the critical length of the domain \cite{beccari_critical_2020,carnicer_critical_2003} must be ensured.
By validating our abstract assumptions, we establish the completeness of hierarchical Tchebycheffian splines as a concrete realization of our general theory; see Corollary~\ref{cor:Tchebycheffian hierachichal complete}.

The paper is organized as follows. Section~\ref{sec:generalized spline spaces} introduces the abstract framework, defining the generalized spline spaces, the contact conditions, and proving the completeness theorem. Section~\ref{sec:tchebycheffian spline spaces} specializes these results to Tchebycheffian splines, reviewing the necessary properties of ECT-systems and weight functions \cite{lyche_tchebycheffian_2019,mazure_constructing_2018} required to apply the general theory.

Finally, we remark that the term \emph{generalized splines} has been often used in the literature, either with a more or a less restrictive meaning with respect to the definition  adopted in this work; see, e.g., \cite{manni_survey_2017, schumaker_spline_2007}.

\section{Generalized spline spaces} \label{sec:generalized spline spaces}

In this section, we define the space of generalized splines on a multi-cell domain and the space of generalized hierarchical splines. We also present a basis construction for these spaces.

\subsection{Definition of generalized spline spaces}

Let $T \coloneqq \{t_i\}_{i=1}^{n}$ be a strictly increasing sequence of interior points of the closed interval $[a,b]\subseteq\R$, with $t_0 \coloneqq a$ and $t_{n+1} \coloneqq b$. On each sub-interval $I_i \coloneqq [t_i,t_{i+1}]$ for $i=0,\dots,n$, we consider a finite-dimensional linear space of functions
\[
    \Uspace_i \coloneqq \Span\{u_{i,j}\}_{j=0}^{p} \subseteq C^{p}(I_i),
\]
for some integer $p\ge1$. The \emph{generalized piecewise space} is defined as
\[
    \PWspace(\{\Uspace_i\}_{i=0}^n , T)
    \coloneqq \Bigl\{ s \coloneqq (s_0,s_1,\dots,s_n) \ \Big|\ s_i\in \Uspace_i \text{ for } i=0,\dots,n \Bigr\}.
\]

\begin{definition}[Generalized spline space]
    Given a vector of integers $\mathbf{r} \coloneqq (r_1,\dots,r_n)$ with $0\le r_i\le p$ for all $i$, the \emph{generalized spline space} associated with $\{\Uspace_i\}_{i=0}^n$, $T$, and $\mathbf{r}$ is
    \[
        \Sspace(\{\Uspace_i\}_{i=0}^n , T , \mathbf{r})
        \coloneqq
        \left\{
        s \in \PWspace(\{\Uspace_i\}_{i=0}^n , T)
        \ \middle|\
        \begin{aligned}
            &D^k s_{i-1}(t_i) = D^k s_i(t_i) \\
            &\forall i=1,\dots,n  \text{ and } \forall k=0,\dots,r_i
        \end{aligned}
        \right\}.
    \]
\end{definition}

We consider a $d$-dimensional box $[\mathbf{a},\mathbf{b}] \coloneqq [a_1,b_1]\times\dots\times[a_d,b_d]\subseteq\R^d$ and, for each direction $j=1,\dots,d$, a strictly increasing sequence of interior points $T^j \coloneqq \{t_i^j\}_{i=1}^{n_j}$ of $[a_j,b_j]$, with $t_0^j \coloneqq a_j$ and $t_{n_j+1}^j \coloneqq b_j$. Given a vector of integers $\mathbf{r}^j$ associated with each $T^j$ for $j=1,\dots,d$, we define the \emph{tensor-product generalized spline space} as
\[
    \Sspace \coloneqq \bigotimes_{j=1}^{d} \Sspace(\{\Uspace_i^j\}_{i=0}^{n_j},\, T^j,\, \mathbf{r}^j),
\]
where each univariate space $\Sspace(\{\Uspace_i^j\}_{i=0}^{n_j},\, T^j,\, \mathbf{r}^j)$ is a generalized spline space on $[a_j,b_j]$ as defined above.

The Cartesian product of the univariate grids $\Ggrid^j \coloneqq \{I_i^j\}_{i=0}^{n_j} \coloneqq \{[t_i^j,t_{i+1}^j]\}_{i=0}^{n_j}$ defines the \emph{multi-dimensional tensor grid}
\[
    \Ggrid \coloneqq \Ggrid^1 \times \dots \times \Ggrid^d.
\]
The cells of $\Ggrid$ are the $d$-dimensional boxes $I_{\mathbf{i}} \coloneqq I^1_{i_1}\times\dots\times I^d_{i_d}$. The boundaries of these cells consist of \emph{faces}. In this framework, a $k$-face (or $k$-box with $k < d$) is a Cartesian product of $d$ sub-intervals, where exactly $d-k$ components are degenerate (i.e., reduced to a single point). We denote by $\Ccell$ the collection of all cells of the grid $\Ggrid$.

Note that the intersection of any two distinct cells $c, c' \in \Ccell$ is either empty or forms a face $\tau = c \cap c'$. The intersection can be represented as a Cartesian product, namely
\[
    \tau = I^1 \times \dots \times I^d,
\]
where each $I^j$ is either a proper interval or degenerate in direction $j$ for $j=1,\dots,d$.

The space $\Sspace$ consists of functions $s$ with pieces $s|_{I_{\mathbf{i}}}\in \Uspace_c \coloneqq \Uspace_{\mathbf{i}} \coloneqq \bigotimes_{j=1}^{d} \Uspace_{i_j}^j$ tied together according to specific regularity constraints, namely the imposition of continuity for successive derivatives with respect to each direction. More generally, one can formulate this condition via the contact of two pieces, similarly to the polynomial contact in \cite{mokris_completeness_2014}.

\begin{definition}[Contact of pieces]
    Let $c$ and $c'$ be two distinct cells of the tensor grid $\Ggrid$ such that their intersection $\tau = c \cap c'$ is non-empty. For each direction $j=1,\dots,d$, let the interval $I^j$ in the Cartesian representation of $\tau$ be denoted by its endpoints, namely
    \[
        I^j = [t^j_{i_1},t^j_{i_2}], \quad i_2 \in \{i_1, i_1+1\}.
    \]
    Note that $I^j$ is degenerate if $i_1 = i_2$.
    We say that $s_c \in \Uspace_c$ and $s_{c'} \in \Uspace_{c'}$ have \emph{contact} on $\tau$, denoted by $s_c \sim s_{c'}$, if for all $\mathbf{x} \coloneqq (x_1,\dots,x_d)\in \tau$,
    \[
        \bigl( \partial_{x_1}^{k_1}\cdots\partial_{x_d}^{k_d}s_c \bigr)(\mathbf{x})
        =
        \bigl( \partial_{x_1}^{k_1}\cdots\partial_{x_d}^{k_d}s_{c'} \bigr)(\mathbf{x})
    \]
    is satisfied for all multi-indices $\mathbf{k} \coloneqq (k_1,\dots,k_d)$ such that for each direction $j=1,\dots,d$,
    \[
        0 \le k_j \le \max \bigl(r^j_{i_1} ,\, r^j_{i_2} \bigr),
    \]
    where $r^j_{i_1}$ and $r^j_{i_2}$ are the prescribed regularity orders associated with the points $t^j_{i_1}$ and $t^j_{i_2}$ in direction $j$. We also say that the pieces $s_c$ and $s_{c'}$ have contact whenever $c \cap c' = \emptyset$ (without conditions).
\end{definition}

Since the contact condition encodes continuity of successive derivatives with respect to each direction, we have the following equivalent characterization of the spline space $\Sspace$ in terms of contact of pieces:
\[
    \Sspace = \Bigl\{s \coloneqq (s_c)_{c\in \Ccell} \ \Big|\ s_c \in \Uspace_c \ \forall c\in \Ccell \text{ and } s_c \sim s_{c'} \ \forall c,c' \in \Ccell \Bigr\},
\]
where $\Ccell$ is the collection of all cells of the grid $\Ggrid$.

\subsection{Generalized spline space on a multi-cell} \label{ssec:space multi-cell}

A multi-cell is a finite collection of cells. Let $\Mcell$ be a multi-cell of $\Ggrid$, i.e., $\Mcell \subseteq \Ccell$. The collection defines a multi-cell domain, $\bigcup_{c \in \Mcell} c$, and to simplify the notation it will be also denoted by $\Mcell$ when no confusion is possible.
The \emph{generalized piecewise space on $\Mcell$} is defined as
    \[
        \PWspace(\Mcell) \coloneqq
        \Bigl\{
        s \coloneqq (s_c)_{c\in \Mcell} \ \Big|\ s_c\in \Uspace_c \ \forall c \in \Mcell
        \Bigr\}.
    \]
We next define the generalized spline space on a multi-cell, introduced in the polynomial setting in \cite{mokris_completeness_2014}.

\begin{definition}[Generalized spline space on a multi-cell]
    Let $\Mcell$ be a multi-cell of $\Ggrid$. The \emph{generalized spline space on $\Mcell$} is defined as
    \[
        \Sspace(\Mcell) \coloneqq
        \Bigl\{
        s \coloneqq (s_c)_{c\in \Mcell} \in \PWspace(\Mcell) \ \Big|\ s_c \sim s_{c'} \ \forall c,c'\in \Mcell
        \Bigr\}.
    \]
\end{definition}

Note that the space $\Sspace(\Mcell)$ is more general than $\Sspace|_{\Mcell}$ since $ \Sspace|_{\Mcell} \subseteq \Sspace(\Mcell) $; see Examples~\ref{EX2.4} and~\ref{EX2.5} below. We can associate a unique function with each generalized spline $s \in \Sspace(\Mcell)$, defined as
\begin{equation} \label{eq:s-function}
    \sum_{c \in \Mcell} s_c(\mathbf{x}) \chi^{*}_c(\mathbf{x})  \text{ for all } \mathbf{x} \in \Mcell,
\end{equation}
with $\chi^{*}_c(\mathbf{x}) \coloneqq \frac{\chi_c(\mathbf{x})}{\sum_{k \in \Mcell}\chi_k(\mathbf{x})} $ the normalized characteristic function. For simplicity of notation, the function in \eqref{eq:s-function} will be also denoted by $s$, and so it makes sense to write $s_c = s|_c$ for all $c\in \Mcell$.

\begin{example} \label{EX2.4}
    Consider the interval $[a, b] = [0, 2]$ with a single interior point at $t_1 = 1$. Let the cells be $c_0 = [0, 1]$ and $c_1 = [1, 2]$, and let the multi-cell be $\Mcell = \{c_0\}$. We impose $C^1$ continuity at $t_1$.
    The local spaces are chosen as
    \[
        \Uspace_0 = \Span\{1, x\}, \quad \Uspace_1 = \Span\bigl\{1, (x-1)^2\bigr\}.
    \]
    Both spaces have dimension $2$.
    \begin{itemize}
        \item \textbf{The spline space on the multi-cell:} Since $\Mcell$ consists of a single cell $c_0$, there are no internal contact conditions. Thus,
        \[
            \Sspace(\Mcell) = \Uspace_0 = \Span\{1, x\}, \quad \dim(\Sspace(\Mcell)) = 2.
        \]
        \item \textbf{The restriction of the global space:} A function $s \in \Sspace$ consists of a pair $(s_0, s_1)$ satisfying $s_0(1) = s_1(1)$ and $s_0'(1) = s_1'(1)$. Observe that for any $v \in \Uspace_1$, $v'(1) = 0$. Consequently, any valid global spline must satisfy $s_0'(1) = 0$, and the restriction space is
        \[
            \Sspace|_{\Mcell} = \bigl\{ u \in \Uspace_0 \ \big|\ u'(1) = 0 \bigr\} = \Span\{1\}, \quad \dim(\Sspace|_{\Mcell}) = 1.
        \]
    \end{itemize}
    We conclude that $\Sspace|_{\Mcell} \subsetneq \Sspace(\Mcell)$.
\end{example}

\begin{example} \label{EX2.5}
    Consider the interval $[a, b] = [0, 3]$ partitioned into three unit intervals $c_0, c_1, c_2$. Let the multi-cell be the disconnected set $\Mcell = \{c_0,c_2\}$. We consider the space of global constants ($C^0$ continuity, local spaces $\Uspace_i = \Span\{1\}$).
    \begin{itemize}
        \item \textbf{The spline space on the multi-cell:} Since $c_0$ and $c_2$ are disjoint, the definition of $\Sspace(\Mcell)$ imposes no contact conditions between them. The values on the disjoint components are independent, so
        \[
            \Sspace(\Mcell) = \bigl\{ (k_0, k_2) \ \big|\ k_0, k_2 \in \R \bigr\}, \quad \dim(\Sspace(\Mcell)) = 2.
        \]
        \item \textbf{The restriction of the global space:} A global spline $s \in \Sspace$ must be constant on the entire interval $[a, b]$. Thus, the values on $c_0$ and $c_2$ must be identical ($k_0 = k_2$). The restriction space is
        \[
            \Sspace|_{\Mcell} = \bigl\{ (k, k) \ \big|\ k \in \R \bigr\}, \quad \dim(\Sspace|_{\Mcell}) = 1.
        \]
    \end{itemize}
    Here, the ``bridge'' cell $c_1$ (which is not in $\Mcell$) enforces a global constraint that is invisible to the local definition $\Sspace(\Mcell)$, resulting in strict inclusion.
\end{example}

In order to construct a basis for the generalized spline space on a multi-cell $\Sspace(\Mcell)$, we make the following assumption on the space $\Sspace$.

\begin{assumption}[Basis] \label{ass:basis}
    The space $\Sspace$ is spanned by a basis $\Bbasis$ that satisfies the following properties:
    \begin{enumerate}[label=(P\arabic*)]
        \item\label{ass:basis P1} The functions in $\Bbasis$ are locally linearly independent: for any $c \in \Ccell$, the elements of $\Bbasis$ restricted to $c$, that do not vanish identically on $c$, are linearly independent.
        \item\label{ass:basis P2} The functions $\beta \in \Bbasis$ have local and compact support: $\supp(\beta) \subseteq [\mathbf{a},\mathbf{b}]$ and $\supp(\beta)$ is of the form $I_1 \times\dots\times I_d$, where $I_j$ is a closed bounded interval.
    \end{enumerate}
\end{assumption}

It is a well-known fact that the spline space with polynomial sections admits a basis with the above properties \cite{lyche_foundations_2018,schumaker_spline_2007}. This is also the case, under suitable assumptions, for spline spaces with sections in ECT-spaces \cite{lyche_tchebycheffian_2019,mazure_constructing_2018,schumaker_spline_2007}.

\begin{example}
    Consider the interval $[a, b] = [0, 2]$ with a single interior point at $t_1 = 1$. Let the cells be $c_0 = [0, 1]$ and $c_1 = [1, 2]$, and impose $C^0$ continuity at $t_1$.
    The local spaces are chosen as
    \[
        \Uspace_0 = \Span\{1, x\}, \quad \Uspace_1 = \Span\{1, f(x) \},
    \]
    where the local function $f$ is defined by
    \[
        f(x) \coloneqq \biggl(x - \frac{3}{2}\biggr)^2 \chi_{[\frac{3}{2},2]}(x).
    \]
    A basis for the spline space $\Sspace$ is given by $\{N_1, N_2, N_3\}$, where
    \[
    \begin{alignedat}{3}
        N_1(x) &\coloneqq \begin{cases} 1 - x & \text{if } x \in [0, 1] \\ 0 & \text{if } x \in [1, 2] \end{cases}
            \quad &\implies \quad \supp(N_1) &= [0, 1];
            \\
        N_2(x) &\coloneqq \begin{cases} x & \text{if } x \in [0, 1] \\ 1 & \text{if } x \in [1, 2] \end{cases}
            \quad &\implies \quad \supp(N_2) &= [0, 2];
        \\
        N_3(x) &\coloneqq \begin{cases} 0 & \text{if } x \in [0, 1] \\ f(x) & \text{if } x \in [1, 2] \end{cases}
            \quad &\implies \quad \supp(N_3) &= \left[\frac{3}{2}, 2\right].
    \end{alignedat}
    \]
The support of $N_3$ is $\left[\frac{3}{2}, 2\right]$. This shows that in a generalized spline space, the support of a basis function does not necessarily align with the grid.
\end{example}

Note that each element $\beta \in \Bbasis$ is of the form $\beta \coloneqq (\beta_c)_{c\in \Ccell} = (\beta|_c)_{c\in \Ccell}$ with $\beta_c \in \Uspace_c$ for all $c\in \Ccell$ and $\beta_c \sim \beta_{c'}$ for all $c,c' \in \Ccell$. For each $c\in \Ccell$ we introduce a subspace of the local tensor-product space $\Uspace_c$:
\[
    \Uspace^*_c \coloneqq \bigl\{u \in \Uspace_c : \exists \ s \in \Sspace \text{ such that } u = s_c \bigr\} \subseteq \Uspace_c.
\]
The following example illustrates a case where the local space $\Uspace_c$ contains functions that cannot be part of any global spline, leading to $\Uspace_c^* \subsetneq \Uspace_c$.

\begin{example} \label{EX2.8}
    Consider the setting of Example~\ref{EX2.4} and take the local function $u \in \Uspace_0$ defined by $u(x) = x$.
    \begin{itemize}
        \item By definition, $u \in \Uspace_0^*$ if and only if it is the restriction of some global spline $s \in \Sspace$. This requires the existence of a neighbor function $v \in \Uspace_1$ that satisfies the $C^1$ contact conditions at the interface $t_1=1$:
        \[
            v(1) = u(1) = 1 \quad \text{and} \quad v'(1) = u'(1) = 1.
        \]
        \item Recall that $\Uspace_1 = \Span\{1, (x-1)^2\}$. For any $v \in \Uspace_1$, we have
        \[
             v'(1) = 0.
        \]
    \end{itemize}
    Since the neighbor space $\Uspace_1$ cannot reproduce the non-zero derivative required by $u(x) = x$, no compatible extension exists. Thus, $u \in \Uspace_0$ but $u \notin \Uspace_0^*$, which shows that $\Uspace_0^* \subsetneq \Uspace_0$.
\end{example}

The set $\Bbasis_c \coloneqq \{\beta_c : \beta \in \Bbasis, \ \beta_c \neq 0 \}$ forms a basis of the space $\Uspace^*_c$ and we obtain a local representation for any $u \in \Uspace^*_c$:
\[
    u(\mathbf{x}) = \sum_{\beta_c \in \Bbasis_c} \lambda_{\beta}^c(u) \beta_c(\mathbf{x}) \text{ for all } \mathbf{x} \in c.
\]
We wish to obtain a link between the contact of pieces defined earlier and the local representation obtained by the global basis. To do so, we make the following assumption regarding the local tensor-product spaces $\Uspace_c$.

\begin{assumption}[Extension] \label{ass:extension}
    Consider any $u_c \in \Uspace_c$ and $u_{c'} \in \Uspace_{c'}$ such that $c \cap c' \neq \emptyset$ and $u_c \sim u_{c'}$. Then, there exists a spline $s \in \Sspace$ such that $s_c = u_c $ and $s_{c'} = u_{c'}$.
\end{assumption}

Roughly speaking, Assumption~\ref{ass:extension} ensures that a locally compatible pair of pieces is extendable to form a global spline on the entire grid.
Hence, it implies that $u_c \in \Uspace_c^*$ and $u_{c'}\in \Uspace_{c'}^*$ for any $u_c \in \Uspace_c$ and $u_{c'} \in \Uspace_{c'}$ such that $c \cap c' \neq \emptyset$ and $u_c \sim u_{c'}$.
Later on, we will show that ECT-spaces satisfy the above extension assumption; see Theorem~\ref{thm:ect_assumption_verification}. The following example illustrates a case where this assumption is not verified.

\begin{example} \label{EX2.10}
    Consider the interval $[a, b] = [-1, 2]$ with 2 interior points at $t_1 = 0$, $t_2 = 1$. Let the cells be $c_0 = [-1,0]$, $c_1 = [0, 1]$, and $c_2 = [1, 2]$.
    The local spaces are chosen as
    \[
        \Uspace_0 =\Uspace_1=\Span\{1, x\}, \quad \Uspace_2 = \Span\bigl\{1, (x-1)^2\bigr\}.
    \]
    All spaces have dimension $2$; see also Example~\ref{EX2.4}.
    We impose $C^1$ continuity at $t_1$ and $t_2$. From Example~\ref{EX2.4} (and Example~\ref{EX2.8}) we know that any element $s \in \Sspace$ is such that $s'(1) = 0$.
    Now, take $c=[-1,0]$ and $c'=[0,1]$, $u_c(x)=x$, $u_{c'}(x)=x$.
    We have $u_c \sim u_{c'}$ but there is no $s \in \Sspace$ such that $s_c = u_c $ and $s_{c'} = u_{c'}$.
    We conclude that Assumption~\ref{ass:extension} does not hold.
\end{example}

\begin{remark}
    Verifying Assumption~\ref{ass:extension} for an arbitrary configuration is non-trivial. Intuitively, the problem of constructing global sections from local data suggests a connection to cohomology groups in cellular sheaf theory \cite{curry2014sheaves}. While a rigorous cohomological treatment is beyond the scope of the paper, this perspective highlights that the obstruction depends on both local spaces and smoothness. For instance, in Example~\ref{EX2.8}, simply reducing the regularity to $C^0$ would eliminate the obstruction and restore the equality of the spaces.
\end{remark}

The following lemma highlights a relevant property of the local representations in terms of bases satisfying Assumption~\ref{ass:extension}.

\begin{lemma} \label{lem:coeff_match}
    Let $c, c' \in \Ccell$ such that $c \cap c' \neq \emptyset$, and let $u_c \in \Uspace_c \ , u_{c'} \in \Uspace_{c'}$ such that $u_c \sim u_{c'}$. Under Assumptions \ref{ass:basis} and \ref{ass:extension}, for all $\beta \in \Bbasis$ such that $\beta_c \neq 0$ and $\beta_{c'} \neq 0$, the coefficients in the local representations of $u_c$ and $u_{c'}$ are identical, i.e.,
    \[
        \lambda_{\beta}^c(u_c) = \lambda_{\beta}^{c'}(u_{c'}).
    \]
\end{lemma}

\begin{proof}
    By Assumption~\ref{ass:extension}, there exists $s \in \Sspace$ such that $s|_c = u_c$ and $s|_{c'} = u_{c'}$. This function admits a unique global representation:
    \[
        s = \sum_{\beta \in \Bbasis} \lambda_{\beta}(s) \, (\beta_c)_{c \in \Ccell}.
    \]
    Restricting this expansion to $c$ and $c'$ yields the local representations of $s_c$ and $s_{c'}$. Since the restriction of the basis $\Bbasis$ to any cell is linearly independent (Property~\ref{ass:basis P1}), the local coefficients are uniquely determined by the global coefficients of $s$. Thus, the coefficient corresponding to $\beta$ is $\lambda_{\beta}(s)$ in both local representations.
\end{proof}

To construct a basis for $\Sspace(\Mcell)$, we decompose the global basis functions with respect to the connected components of their support within the domain $\Mcell$. We first define the notion of coefficient graph for our generalized setting. An adapted version was already introduced for the polynomial setting in \cite{mokris_completeness_2014}.

\begin{definition}[Coefficient graph]
    For a global basis function $\beta \in \Bbasis$ and a multi-cell $\Mcell$, the coefficient graph $\Ggraph_{\beta}(\Mcell)$ is defined as follows:
    \begin{itemize}
        \item the vertices are the cells $c \in \Mcell$ such that $\supp(\beta) \cap \overset{\circ}{c} \neq \emptyset$;
        \item an edge connects two vertices if the corresponding cells share a common face.
    \end{itemize}
\end{definition}

Following \cite{mokris_completeness_2014}, we use the notation $c \in \Ggraph_{\beta}(\Mcell)$ to indicate that the cell $c$ is a vertex of $\Ggraph_{\beta}(\Mcell)$.
Let $K(\Ggraph_{\beta})$ be the set of connected components of $\Ggraph_{\beta}(\Mcell)$. For each connected component $\Phi \in K(\Ggraph_{\beta})$, Lemma~\ref{lem:coeff_match} implies that the coefficient of $\beta$ must be constant across all cells in $\Phi$. We define $\beta_{\Phi} \in \Sspace(\Mcell)$ for each $\beta \in \Bbasis$ and $\Phi \in K(\Ggraph_{\beta})$ as
\[
    \beta_{\Phi} \coloneqq (\beta_{\Phi,c})_{c\in \Mcell},
\]
where $\beta_{\Phi,c} \coloneqq \beta_c$ if $c\in \Phi$ and $0$ otherwise. Note that $\beta_{\Phi}$ is well defined and belongs to $\Sspace(\Mcell)$ because the contact conditions are satisfied.
In particular, we have $\beta_{\Phi} = \mathbf{0}$ if $\supp(\beta) \cap \overset{\circ}{\Mcell} = \emptyset$, i.e., $\Ggraph_{\beta}(\Mcell)$ is empty. We are now ready to provide a basis for $\Sspace(\Mcell)$. 

\begin{remark}
Let $\Mcell$ be a multi-cell and suppose that it contains a cell $c$ that does not share a face with any other cell of $\Mcell$. Then, no contact conditions are imposed on the piece $s_c$. Consequently, the value of $s_c$ can be chosen arbitrarily in the local space $\Uspace_c$, and the space $\Sspace(\Mcell)$ decomposes as the direct sum
\[
\Sspace(\Mcell)
=
\Sspace(\Mcell\setminus\{c\}) \oplus \Uspace_c .
\]
Therefore, constructing a basis for $\Sspace(\Mcell)$ reduces to constructing a basis for $\Sspace(\Mcell\setminus\{c\})$.
\end{remark}

Without loss of generality, we restrict our attention to multi-cells without isolated cells, i.e., multi-cells in which every cell shares a face with at least one other cell.

\begin{theorem}[Basis for $\Sspace(\Mcell)$]\label{thm:basis D(M)}
    Under Assumptions~\ref{ass:basis} and~\ref{ass:extension}, the set
    \[
        \Dbasis(\Mcell) \coloneqq \Bigl\{ \beta_{\Phi} \coloneqq (\beta_{\Phi,c})_{c \in \Mcell} \ \Big|\ \beta \in \Bbasis, \ \supp(\beta) \cap  \overset{\circ}\Mcell \neq \emptyset,\ \Phi \in K(\Ggraph_{\beta}) \Bigr\}
    \]
    forms a basis for the generalized spline space $\Sspace(\Mcell)$.
\end{theorem}

\begin{proof}
    \textbf{Linear independence:} Consider a linear combination equal to zero:
    \[
        \sum_{\beta_{\Phi} \in \Dbasis(\Mcell)}\alpha_{\beta,\Phi} \beta_{\Phi} =\sum_{\beta \in \Bbasis, \; \supp(\beta) \cap \overset{\circ}{\Mcell} \: \neq \: \emptyset }  \sum_{\Phi \in K(\Ggraph_{\beta})} \alpha_{\beta,\Phi} \beta_{\Phi} = \mathbf{0} = (0)_{c \in \Mcell}.
    \]
    This vector equality implies equality for every component associated with $c \in \Mcell$. Restricting the sum to the $c$-th component, we get
    \[
        \sum_{\beta \in \Bbasis, \; \supp(\beta) \cap \overset{\circ}{\Mcell} \: \neq \: \emptyset} \sum_{\Phi \in K(\Ggraph_{\beta})} \alpha_{\beta,\Phi} \beta_{\Phi,c} = 0.
    \]
    For a fixed $c$ and a fixed $\beta$, there is at most one connected component $\Phi \in K(\Ggraph_{\beta})$ such that $c \in \Phi$. Let us denote this unique connected component by $\Phi(\beta, c)$ (if it exists). The sum simplifies to
    \[
        \sum_{\beta \in \Bbasis, \; \supp(\beta) \cap \overset{\circ}{c} \: \neq \: \emptyset} \alpha_{\beta,\Phi(\beta,c)} \beta|_{c} = 0.
    \]
    By the local linear independence of the basis $\Bbasis$ restricted to $c$ (Assumption~\ref{ass:basis}), we must have $\alpha_{\beta,\Phi(\beta,c)} = 0$ for all active $\beta$. Since this holds for every cell $c$, all coefficients $\alpha_{\beta,\Phi}$ must be zero.

    \medskip
    \noindent \textbf{Spanning:}
    Let $s \coloneqq (s_c)_{c \in \Mcell} \in \Sspace(\Mcell)$. On any cell $c \in \Mcell$, the local piece $s_c$ admits a unique expansion in the restricted basis:
    \[
        s_c = \sum_{\beta_c \in \Bbasis_c} \lambda_{\beta}^c \beta_{c}.
    \]
    If two cells $c, c'$ belong to the same connected component $\Phi \in K(\Ggraph_{\beta})$, there exists a path of adjacent cells connecting them. Repeated application of Lemma~\ref{lem:coeff_match} ensures that $\lambda_{\beta}^c = \lambda_{\beta}^{c'}$.
    Thus, for each $\beta$ and each component $\Phi$, we can assign a unique coefficient $\Lambda_{\beta,\Phi}$. We then construct the sum
    \[
        \sigma \coloneqq \sum_{\beta \in \Bbasis} \sum_{\Phi \in K(\Ggraph_{\beta})} \Lambda_{\beta,\Phi} \beta_{\Phi}.
    \]
    Taking the $c$-th component of $\sigma$, we get
    \[
        (\sigma)_c = \sum_{\beta \in \Bbasis} \Lambda_{\beta,\Phi(\beta,c)} \beta_{\Phi,c} = \sum_{\beta_c \in \Bbasis_c} \lambda_{\beta}^c \beta_c = s_c.
    \]
    Hence, $\sigma = s$, and the set spans the space.
\end{proof}

We aim to construct a basis for the generalized spline space $\Sspace(\Mcell)$ by restricting the functions from $\Bbasis$ to the multi-cell $\Mcell$. To ensure that this restriction process yields a well-defined basis, without the need for the basis $\Dbasis(\Mcell)$ introduced in Theorem~\ref{thm:basis D(M)}, we focus our attention on multi-cell domains that preserve the connectivity of the basis supports.

\begin{assumption}[Support connectivity] \label{ass:connectivity}
    For every basis function $\beta \in \Bbasis$, the intersection of its support with the multi-cell, $\supp(\beta) \cap \overset{\circ}{\Mcell}$, is connected.
\end{assumption}

\begin{corollary} \label{cor:basis D(M)}
    Under Assumptions~\ref{ass:basis},~\ref{ass:extension}, and~\ref{ass:connectivity}, the set
    \begin{equation}
    	\label{eq:basis-support-connectivity}
       \Bigl\{ \beta|_{\Mcell} = (\beta|_{c})_{c \in \Mcell} \ \Big|\ \beta \in \Bbasis, \ \supp(\beta) \cap \overset{\circ}{\Mcell} \neq \emptyset \Bigr\}
    \end{equation}
    forms a basis for the generalized spline space $\Sspace(\Mcell)$.
\end{corollary}

\begin{proof}
    Under Assumption~\ref{ass:connectivity}, the coefficient graph $\Ggraph_{\beta}(\Mcell)$ consists of exactly one connected component for every active $\beta \in \Bbasis$.
    Consequently, the basis $\Dbasis(\Mcell)$ in Theorem~\ref{thm:basis D(M)} is composed of the non-trivial restrictions to $\Mcell$ of the global basis functions in $\Bbasis$, as given in \eqref{eq:basis-support-connectivity}.
\end{proof}

\subsection{Generalized hierarchical spline space} \label{ssec:hierarchical space}

We build further on the results of the previous sections to address now the hierarchical setting, which is of particular interest in applications.

We consider a sequence of tensor-product generalized spline spaces $\{\Sspace_l\}_{l=0}^{N-1}$ with the associated $d$-dimensional grids $\{\Ggrid_l\}_{l=0}^{N-1}$. We denote by $\Ccell_l$ the set of all cells of the grid $\Ggrid_l$ and by $\Fface_l$ the set of all faces of $\Ggrid_l$.

\begin{definition}[Multi-level multi-cell]
    Let $\Gamma_l$ be a multi-cell of the grid $\Ggrid_l$ for $l=0,\dots,N-1$ such that
    \begin{itemize}
        \item $\forall \ l,k : \overset{\circ}{\Gamma}_k \cap \overset{\circ}{\Gamma}_l = \emptyset $;
        \item $\forall \ l \le k : \partial\Gamma_k \cap \partial\Gamma_l = \bigcup_{i \in \mathfrak{I}} \tau_{k,i} $ with $\mathfrak{I}$ being a finite set and $\tau_{k,i} \in \Fface_k \ \forall i \in \mathfrak{I}$.
    \end{itemize}
    We define the \emph{multi-level multi-cell} $\Gamma \coloneqq \bigcup_{l=0}^{N-1} \Gamma_l$ as a collection of cells from the levels $l=0,\dots,N-1$ satisfying the two conditions above. The collection $\Gamma$ defines a domain $\bigcup_{c\in \Gamma} c$, called a multi-level domain, and to simplify the notation it will be also denoted by $\Gamma$ when no confusion is possible.
\end{definition}

An example of a valid multi-level multi-cell and an invalid configuration is depicted in Figure~\ref{fig:multi-level multi-cell}.

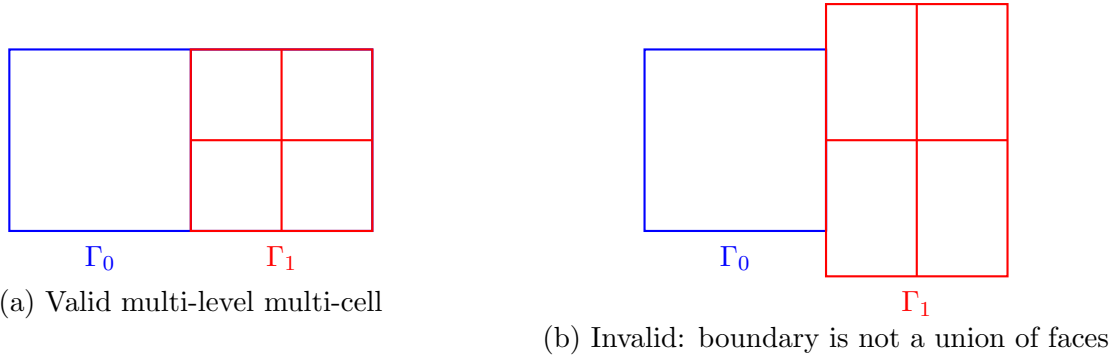
\begin{figure}[h]
\centering
\begin{tikzpicture}[scale=1.2]

\begin{scope}[xshift=0cm]

\draw[thick, blue] (0,0) rectangle (4,2);
\draw[thick, blue] (2,0) -- (2,2);
\node at (1,-0.3) [blue]{$\Gamma_0$};
\draw[thick, red] (2,0) rectangle (4,2);
\draw[thick, red] (2,1) -- (4,1);
\draw[thick, red] (3,0) -- (3,2);
\node at (3,-0.3) [red]{$\Gamma_1$};
\node at (2,-0.8) {(a) Valid multi-level multi-cell};

\end{scope}

\begin{scope}[xshift=7cm]

\draw[thick, blue] (0,0) rectangle (2,2);
\draw[thick, blue] (2,0) -- (2,2);
\node at (1,-0.3) [blue]{$\Gamma_0$};
\draw[thick, red] (2,-0.5) rectangle (4,2.5);
\draw[thick, red] (2,1) -- (4,1);
\draw[thick, red] (3,-0.5) -- (3,2.5);
\node at (3,-0.8) [red]{$\Gamma_1$};
\node at (2,-1.2) {(b) Invalid: boundary is not a union of faces};

\end{scope}

\end{tikzpicture}

\caption{
(a) Valid configuration: interiors are disjoint and
$\partial\Gamma_1 \cap \partial\Gamma_0$ is a union of level-$1$ faces.
(b) Invalid configuration: interiors remain disjoint, but
the boundary intersection consists of portions of vertical level-$1$ edges,
hence it is not a union of faces of the finer grid.
} \label{fig:multi-level multi-cell}
\end{figure}

We aim to define a generalized spline space on a multi-level multi-cell, similarly to the generalized spline space on a multi-cell, using the spaces $\Sspace_l(\Gamma_l) $ for $l=0,\dots,N-1$.
If the multi-level domain consists of multiple disjoint connected components, the resulting spline space is the direct sum of the spaces associated with each component. Consequently, the analysis can be reduced to connected multi-level domains.
Without loss of generality, we assume that the multi-level domain is connected.

The next definition extends the contact condition to the multi-level framework.

\begin{definition}[Hierarchical contact]
    Let $\Gamma \coloneqq \bigcup_{l=0}^{N-1} \Gamma_l$ be a multi-level multi-cell, and for each level $l$, $s_l \coloneqq (s_{l,c_l})_{c_l \in \Gamma_l} \in \Sspace_l(\Gamma_l)$.
    For two levels $l < k$, we say that $s_l$ and $s_k$ have \emph{hierarchical contact}, and we write $s_l \sim_{\h} s_k$, if the following holds:
    \[
        \forall\, c_l \in \Gamma_l,\; \forall\, c_k \in \Gamma_k
        \text{ such that } c_l \cap c_k = \tau_k \in \Fface_k,
        \quad
        s_{l,c_l} \sim s_{k,c_k}
        \text{ with respect to level } k. 
    \]
    Here, $\Fface_k$ denotes the set of faces of the level-$k$ grid, and ``with respect to level $k$'' means that the contact order is dictated by the finer (level-$k$) regularity.
\end{definition}

We now introduce the piecewise space on a multi-level multi-cell $\Gamma \coloneqq \bigcup_{l=0}^{N-1} \Gamma_l $:
\[
    \PWspace(\Gamma) \coloneqq
    \Bigl\{
    s \coloneqq (s_l)_{l= 0}^{N-1} \ \Big|\ s_l\in \Sspace_l(\Gamma_l) \ \text{for} \ l=0,\dots,N-1
    \Bigr\}.
\]
The hierarchical contact and the piecewise space on a multi-level multi-cell allow us to naturally define the generalized spline space on a multi-level multi-cell, which is in fact another way to define a hierarchical spline space on a hierarchical domain that coincides with the multi-level multi-cell.

\begin{definition}[Generalized hierarchical spline space]
    The \emph{generalized hierarchical spline space} on the multi-level multi-cell $\Gamma \coloneqq \bigcup_{l=0}^{N-1}\Gamma_l$ is defined as
    \[
        \Hspace(\Gamma) \coloneqq
        \Bigl\{
        s \coloneqq (s_l)_{l=0}^{N-1}\in \PWspace(\Gamma) \ \Big|\ s_l \sim_{\h} s_k
        \ \forall l<k
        \Bigr\}.
    \]
    Hence, $\Hspace(\Gamma)$ contains all collections of level-wise spline functions that satisfy hierarchical contact.
\end{definition}

\subsection{Generalized hierarchical spline basis} \label{ssec:hierarchical basis}

In order to construct a basis for the generalized hierarchical spline space $\Hspace(\Gamma)$, we make the following assumption.

\begin{assumption}[Hierarchy of spline spaces] \label{ass:basis hierachichal setting}
    The sequence $\{\Sspace_l\}_{l=0}^{N-1}$ satisfies the following:
    \begin{enumerate}
        \item the sequence is defined on a common domain $\Omega \supseteq \Gamma$;
        \item the sequence is nested ($\Sspace_{0}\subseteq \dots \subseteq \Sspace_{N-1}$);
        \item each space $\Sspace_l$ is spanned by a basis $\Bbasis_l$ that satisfies Assumption~\ref{ass:basis};
        \item each space $\Sspace_l$ satisfies Assumption~\ref{ass:extension}.
    \end{enumerate}
\end{assumption}

We define the generalized hierarchical spline basis via a selection mechanism, which is similar to the usual construction of the hierarchical B-spline basis \cite{giannelli_strongly_2014,lyche_foundations_2018}. From each level $l$, we select the functions in $\Bbasis_l$ that do not vanish on $\Gamma_l$ and with disjoint support from the interior of $\Gamma_k$ for $k<l$.

\begin{definition}[Selection mechanism]\label{def:selection}
    For each level $l$, we define the set $\Hbasis_l \subseteq \Bbasis_l$ as
    \[
        \Hbasis_l \coloneqq \Bigl\{ \beta \in \Bbasis_l \ \Big|\ \beta|_{\Gamma_l} \neq 0 \ \wedge \ \supp(\beta) \cap \overset{\circ}{\Gamma}_k = \emptyset \:\: \forall k < l \bigr\}.
    \]
\end{definition}

The selection mechanism allows us to construct a set of linearly independent functions in $\Hspace(\Gamma)$, which, under suitable assumptions, turns out to span the generalized hierarchical spline space, as shown in the following theorems.
 
\begin{theorem}[Generalized hierarchical spline basis] \label{thm:linind}
   Define the set
   \begin{equation}
   	\label{eq:hierarchical-basis}
        \Hbasis_{\Gamma} \coloneqq  \biggl\{ \beta_{\Gamma} \coloneqq (\beta|_{\Gamma_l})_{l=0}^{N-1} \ \bigg|\ \beta \in
\bigcup_{l=0}^{N-1} \Hbasis_l \biggr\}.
\end{equation}
Under Assumption~\ref{ass:basis hierachichal setting}, we have $\Hbasis_{\Gamma} \subseteq \Hspace(\Gamma)$ and $\Hbasis_{\Gamma}$ is linearly independent.
\end{theorem}

\begin{proof}
    Let us fix $l \in \{0,\dots,N-1\}$ and let $\beta \in \Hbasis_l $. We first remark that the components of $\beta_{\Gamma} \in \Hbasis_{\Gamma}$ are well defined for $k \ge l$  because of the nestedness of the spaces and they vanish  for $k<l$. Moreover, since $\beta|_{\Gamma_k}$ and $\beta|_{\Gamma_l}$ for $l<k$ are restrictions of the same function $\beta \in \Sspace_l$, $\beta_{\Gamma}$ naturally satisfies the hierarchical contact condition $\beta|_{\Gamma_k} \sim_{\h} \beta|_{\Gamma_l}$. This proves $\Hbasis_{\Gamma} \subseteq \Hspace(\Gamma)$.

    Let $\sum_{\beta_{\Gamma} \in \Hbasis_{\Gamma}} c_{\beta_{\Gamma}} \beta_{\Gamma} = 0 \in \Hspace(\Gamma)$. We verify linear independence level-by-level. We consider the first component corresponding to $\Gamma_0$. Then, only $\beta_{\Gamma}$ with $\beta \in\Hbasis_0$ are non-zero (since for $l > 0$ it holds $\beta|_{\Gamma_0} = 0$):
    \[
        \sum_{\beta \in \Hbasis_0} c_{\beta_{\Gamma}} \beta|_{\Gamma_0} = 0 \in \Sspace_0(\Gamma_0).
    \]
    By the local linear independence of $\Bbasis_0$ (Assumption~\ref{ass:basis}), we must have $c_{\beta_{\Gamma}} = 0$ for all $\beta \in \Hbasis_0$.

    Proceeding inductively: if the coefficients are zero for all $\beta_{\Gamma}$ up to level $l-1$, the component corresponding to $\Gamma_l$ contains only terms from $\Hbasis_l$. By the local linear independence of $\Bbasis_l$, these coefficients must also be zero. Thus, $\Hbasis_{\Gamma}$ is linearly independent.
\end{proof}

\begin{theorem}[Completeness of the hierarchical spline basis] \label{thm:hier_complete}
    Under the hypothesis of Theorem~\ref{thm:linind} and if Assumption~\ref{ass:connectivity} is satisfied for each level, the generalized hierarchical spline basis $\Hbasis_{\Gamma}$ in \eqref{eq:hierarchical-basis} is complete, i.e., $\Span\Hbasis_{\Gamma} = \Hspace(\Gamma)$.
\end{theorem}

\begin{proof}
    Let $s \coloneqq (s_0, \dots, s_{N-1}) \in \Hspace(\Gamma)$.
    We construct the representation inductively.
    On $\Gamma_0$, $s_0$ belongs to the restricted space $\Sspace_0|_{\Gamma_0} $. Since $\Bbasis_0$ spans $\Sspace_0$, there exist unique coefficients for the functions in $\Hbasis_0$ (which are the only active functions on $\Gamma_0$) to represent $s_0$.
    Assume we have represented $s$ up to level $l-1$. Consider level $l$, the error between $s_l$ and the extension of the coarser basis functions is the residual $f_l$.
    Due to the nestedness, the coarse basis functions extend exactly into the space $\Sspace_l$. Due to the hierarchical contact $s_{l-1} \sim_{\h} s_l$, the residual $f_l$ vanishes on the boundary shared with coarser levels.
    Therefore, $f_l$ lies in the subspace of $\Sspace_l$ spanned by functions whose supports do not intersect coarser levels; exactly the set $\Hbasis_l$. Since $\Bbasis_l$ spans $\Sspace_l$ (Corollary~\ref{cor:basis D(M)}), we can find coefficients for $\Hbasis_l$ to match this residual exactly.
\end{proof}

\subsection{Generalized decoupled hierarchical spline basis}
\label{ssec:decoupled basis}

We now address the situation where Assumption~\ref{ass:connectivity} is no longer satisfied. In this context, we can construct a hierarchical basis by using the bases $\Dbasis(\Gamma_l)$ introduced in Theorem~\ref{thm:basis D(M)}. However, simply selecting functions is not sufficient. While a global basis function $\beta \in \Bbasis_l$ can be extended to the finer space $\Sspace_{l+1}$ due to nestedness, the elements of the basis $\Dbasis(\Gamma_l)$ might not, since the equality $\Sspace(\Gamma_l)=\Sspace|_{\Gamma_l}$ is no longer guaranteed. We must ensure that the ``cut'' performed to isolate a component does not violate the hierarchical contact. This requires the \emph{feasibility of decoupling} condition, introduced for the polynomial case in \cite{mokris_completeness_2014}. In the following definition, each connected component $\Phi$ of the coefficient graph represents a domain $\bigcup_{c\in \Phi} c$.

\begin{definition}[Feasibility of decoupling]\label{def:feas of decoup}
    Let $\beta \in \Bbasis_l$ be a basis function with disconnected coefficient graph on $\Gamma_l$. The decoupling of $\beta$ is said to be \emph{feasible} with respect to the finer space $\Sspace_{l+1}$ if for every fine basis function $\gamma \in \Bbasis_{l+1}$, the support of $\gamma$ intersects the interior of at most one connected component $\Phi \in K(\Ggraph_{\beta})$.
\end{definition}

\begin{remark} \label{rmk:decoupling}
A sufficient condition for the feasibility of decoupling in some level $l$ is that the refinement in level $l+1$ is dense enough with respect to the grid $\Ggrid_l$. Specifically, we require that for each element $\gamma$ of the fine basis $\Bbasis_{l+1}$, the support is contained within a multi-cell $\Mcell_{\gamma}$ of the coarser level $l$:
\[\supp(\gamma) \subseteq \Mcell_{\gamma},\]
where $\Mcell_{\gamma}$ consists of at most two adjacent cells of $\Ccell_l$.
This geometric constraint ensures that the support of a fine basis function $\gamma$ is sufficiently localized. Consequently, it cannot bridge the gap between two disconnected components of a coarse basis function's support, thereby guaranteeing that $\supp(\gamma)$ intersects at most one component $\Phi \in K(\Ggraph_{\beta})$ for any $\beta \in \Bbasis_l$.
\end{remark}

Figure~\ref{fig:decoupling feasibility} illustrates this condition. The multi-level domain $\Gamma$ (top) is constructed such that the supports of the basis functions at each level (bottom) are sufficiently small relative to the coarse cells they inhabit.

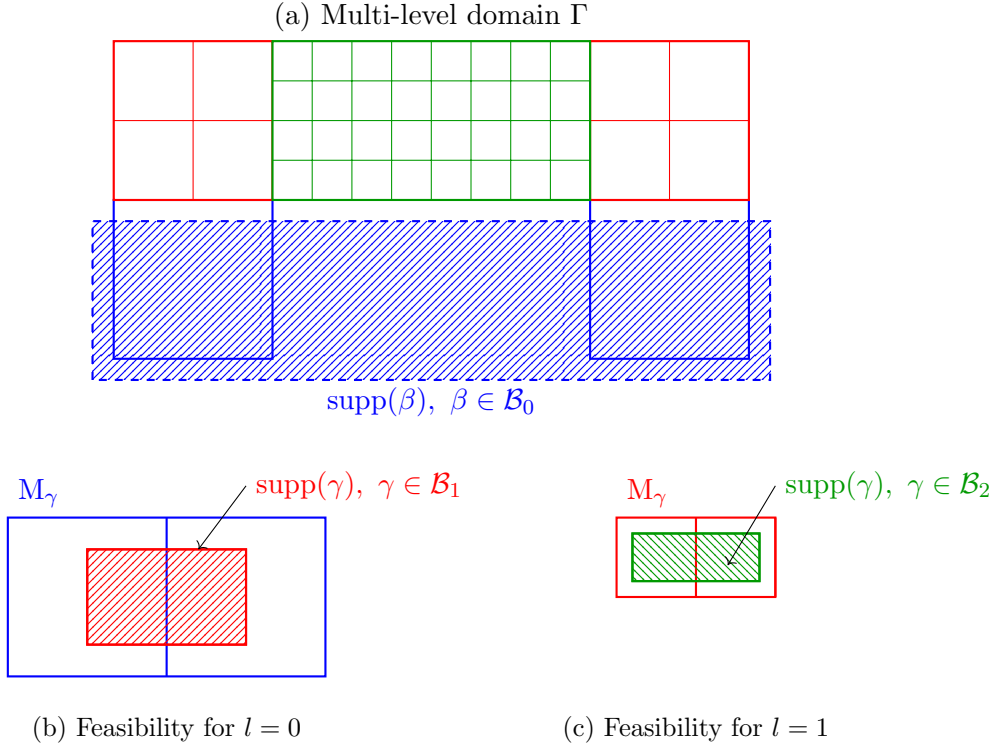
\begin{figure}[h]
    \centering
    \begin{tikzpicture}[scale=1.4]

        \begin{scope}[yshift=3cm, xshift=1cm]
            \def\sz{1.5}

            \fill[pattern=north east lines, pattern color=blue] (-0.2, -0.2) rectangle (4*\sz + 0.2, \sz - 0.2);
            \draw[blue, thick, dashed] (-0.2, -0.2) rectangle (4*\sz + 0.2, \sz - 0.2);

            \node[blue, fill=white, inner sep=2pt, anchor=south] at (2*\sz, -0.6) {$\supp(\beta),\ \beta \in \Bbasis_0$ };

            \draw[thick, blue] (0, 0) rectangle (\sz, \sz);

            \draw[thick, red, fill=white] (0, \sz) rectangle (\sz, 2*\sz);
            \draw[red, step=\sz/2] (0, \sz) grid (\sz, 2*\sz);

            \draw[thick, blue] (3*\sz, 0) rectangle (4*\sz, \sz);

            \draw[thick, red, fill=white] (3*\sz, \sz) rectangle (4*\sz, 2*\sz);
            \draw[red, step=\sz/2] (3*\sz, \sz) grid (4*\sz, 2*\sz);

            \draw[thick, green!60!black, fill=white] (\sz, \sz) rectangle (3*\sz, 2*\sz);
            \draw[green!60!black, step=\sz/4] (\sz, \sz) grid (3*\sz, 2*\sz);

            \node[black, anchor=south] at (2*\sz, 2*\sz) {(a) Multi-level domain $\Gamma$};
        \end{scope}

        \begin{scope}[yshift=0cm, xshift=0cm]
            \def\cz{1.5}

            \draw[thick, blue] (0,0) rectangle (2*\cz, \cz);
            \draw[thick, blue] (\cz,0) -- (\cz, \cz);
            \node[blue, anchor=south west] at (0, \cz) {$\Mcell_\gamma$};

            \draw[thick, red, pattern=north east lines, pattern color=red] (0.5*\cz, 0.2*\cz) rectangle (1.5*\cz, 0.8*\cz);
            \draw[thick, red] (0.5*\cz, 0.2*\cz) rectangle (1.5*\cz, 0.8*\cz);

            \draw[<-] (1.2*\cz, 0.8*\cz) -- (1.5*\cz, 1.2*\cz) node[anchor=west, red] {$\supp(\gamma),\ \gamma \in \Bbasis_1$};

            \node[black, font=\small] at (\cz, -0.5) {(b) Feasibility for $l=0$};
        \end{scope}

        \begin{scope}[yshift=0cm, xshift=5cm]
            \def\cz{1.5}

            \draw[thick, red] (0.5*\cz,0.5*\cz) rectangle (1.5*\cz, \cz);
            \draw[thick, red] (\cz,0.5*\cz) -- (\cz, \cz);
            \draw[thick, red] (1.5*\cz,0.5*\cz) -- (1.5*\cz, \cz);
            \node[red, anchor=south west] at (0.5*\cz, \cz) {$\Mcell_\gamma$};

            \draw[thick, green!60!black, pattern=north west lines, pattern color=green!60!black] (0.6*\cz, 0.6*\cz) rectangle (1.4*\cz, 0.9*\cz);
            \draw[thick, green!60!black] (0.6*\cz, 0.6*\cz) rectangle (1.4*\cz, 0.9*\cz);

            \draw[<-] (1.20*\cz, 0.7*\cz) -- (1.50*\cz, 1.2*\cz) node[anchor=west, green!60!black] {$\supp(\gamma),\ \gamma \in \Bbasis_2$};

            \node[black, font=\small] at (\cz, -0.5) {(c) Feasibility for $l=1$};
        \end{scope}

    \end{tikzpicture}
    \caption{(a) The multi-level domain $\Gamma$. (b) The support of a level-$1$ basis function (red) is contained within a multi-cell consisting of two adjacent cells from level $0$ (blue). (c) The support of a level-$2$ basis function (green) is contained within a multi-cell consisting of two adjacent cells from level $1$ (red).}
    \label{fig:decoupling feasibility}
\end{figure}

\begin{remark}
    Examples of refinement strategies that satisfy the sufficient condition in Remark~\ref{rmk:decoupling} for polynomial splines can be found in \cite{mokris_completeness_2014}.
\end{remark}

If the feasibility of the decoupling holds for some level $l$, each $\beta_{\Phi} \in \Dbasis(\Gamma_l)$ can be extended to $\Sspace_{l+1}$ as follows:
\begin{equation}
\label{eq:extension}
    E_l(\beta_{\Phi}) \coloneqq \sum_{\gamma \in \Bbasis_{l+1} \: : \: \supp(\gamma) \cap \overset{\circ}{\Phi} \: \neq \: \emptyset} \lambda_{\gamma}^{l+1}(\beta) \ \gamma ,
\end{equation}
where $\lambda_{\gamma}^{l+1}(\beta)$ is the coefficient in the representation of $\beta \in \Bbasis_l$ with respect to the basis $\Bbasis_{l+1}$.

\begin{lemma} \label{lem:extension}
The extension in \eqref{eq:extension} is consistent, i.e.,
\[
E_l(\beta_{\Phi})|_{\Gamma_l} = \beta_{\Phi} \in \Sspace_l(\Gamma_l).
\]
\end{lemma}

\begin{proof}
Since the spaces are nested, each $\beta \in \Bbasis_l$
admits a unique representation in $\Sspace_{l+1}$:
\[
\beta
=
\sum_{\gamma \in \Bbasis_{l+1}}
\lambda^{l+1}_\gamma(\beta)\,\gamma.
\]
By the feasibility of the decoupling, for every $\gamma \in \Bbasis_{l+1}$ the support of $\gamma$ intersects at most one connected component of
$\Ggraph_\beta(\Gamma_l)$.  Let $\Phi$ be a connected component of $\Ggraph_\beta(\Gamma_l)$.

We now compare $E_l(\beta_{\Phi})$ and $\beta_{\Phi}$ on $\Gamma_l$.
If $c \in \Phi$, then on $c$ the function $\beta_{\Phi}$ coincides with $\beta$.
Moreover, on $c$ the representation of $\beta$ in terms of $\Bbasis_{l+1}$ reduces precisely to the sum over those $\gamma$ whose support intersects $\Phi$.
Hence,
\[
E_l(\beta_{\Phi})|_c
=
\beta_c
=
\beta_{\Phi,c}.
\]
If $c \in \Gamma_l$ and $c \notin \Phi$, then $\beta_{\Phi,c} = 0$.
By the feasibility, any $\gamma$ appearing in the expansion
defining $E_l(\beta_\Phi)$ has support intersecting only $\Phi$,
hence its restriction to $c$ vanishes.
Therefore,
\[
E_l(\beta_{\Phi})|_c = 0 = \beta_{\Phi,c}.
\]
Thus, $E_l(\beta_{\Phi}) |_{\Gamma_l} = \beta_{\Phi} $ and  the extension is consistent.
\end{proof}

We now define a decoupled selection mechanism, again with the aim of constructing a basis for $\Hspace(\Gamma)$.

\begin{definition}[Decoupled selection mechanism]
    For each level $l$, assuming feasibility holds, we select the functions $\beta_{\Phi} \in\Dbasis(\Gamma_l)$ such that the corresponding $\beta$ have disjoint support from the interior of $\Gamma_k$ for $k<l$:
    \[
        \Hbasis_l^{\dec} \coloneqq \Bigl\{ \beta_{\Phi} \in \Dbasis(\Gamma_l) \ \Big|\ \supp(\beta) \cap \overset{\circ}{\Gamma}_k = \emptyset \:\: \forall k < l \Bigr\}.
    \]
\end{definition}

\begin{theorem}[Generalized decoupled hierarchical spline basis] \label{thm:decoupling linind}
	 Define the set
	\begin{equation}
		\label{eq:decoulpled-hierachical-basis}
		 \Hbasis_{\Gamma}^{\dec} \coloneqq \biggl\{ \beta_{\Gamma}^{\dec} \coloneqq (\beta_{\Phi,l})_{l=0}^{N-1} \ \bigg|\ \beta_{\Phi} \in \bigcup_{l=0}^{N-1} \Hbasis_l^{\dec} \biggr\},
	\end{equation}
where
for each $\beta_{\Phi} \in \Dbasis(\Gamma_l)$:
\begin{itemize}
	\item $\beta_{\Phi,k} \coloneqq 0$ for $k < l$,
	\item $\beta_{\Phi,l} \coloneqq \beta_{\Phi}$,
	\item $\beta_{\Phi,k} \coloneqq E_l(\beta_{\Phi})|_{\Gamma_{k}}$ for $k > l$.
	\end{itemize}	
    Under Assumption~\ref{ass:basis hierachichal setting} and if the decoupling is feasible for each level $l$, we have $ \Hbasis_{\Gamma}^{\dec} \subseteq \Hspace(\Gamma)$ and $ \Hbasis_{\Gamma}^{\dec}$ is linearly independent.
\end{theorem}

\begin{proof}
    Let us fix $l \in \{0,\dots,N-1\}$ and let $\beta_{\Phi} \in \Hbasis_l^{\dec} $.
    We first remark that the components of $\beta_{\Gamma}^{\dec} \in \Hbasis_{\Gamma}^{\dec}$ for $k \ge l$ are well defined because of the nestedness of the spaces and the consistency of the extension (Lemma~\ref{lem:extension}). Since $\beta_{\Phi,k}$ and $\beta_{\Phi,l}$ for $l<k$ are restrictions of the same global function $ E_l(\beta_{\Phi})$, $\beta_{\Gamma}^{\dec}$ naturally satisfies the hierarchical contact condition $\beta_{\Phi,l} \sim_{\h} \beta_{\Phi,k}$. This proves $\Hbasis_{\Gamma}^{\dec} \subseteq \Hspace(\Gamma)$.

    The proof of the linear independence follows the same inductive argument as the proof of Theorem~\ref{thm:linind}.
\end{proof}

\begin{theorem}[Completeness of the decoupled hierarchical spline basis]
\label{thm:decoupling complete}
    Under the hypothesis of Theorem~\ref{thm:decoupling linind},
    the generalized decoupled hierarchical spline basis $\Hbasis_{\Gamma}^{\dec}$ is complete, i.e., $\Span \Hbasis_{\Gamma}^{\dec} = \Hspace(\Gamma)$.
\end{theorem}

\begin{proof}
    The proof follows the same inductive argument as the proof of Theorem~\ref{thm:hier_complete}.
\end{proof}

\section{Tchebycheffian spline spaces} \label{sec:tchebycheffian spline spaces}

We now elaborate the results from Section~\ref{sec:generalized spline spaces} for the relevant class of splines that can be constructed by taking the spaces $\Uspace_i$ to be (extended complete) Tchebycheff spaces.

\subsection{Preliminaries on ECT-spaces}

In this section, we present the class of extended complete Tchebycheff spaces (ECT-spaces), a natural generalization of the polynomial setting \cite{lyche_tchebycheffian_2019,mazure_finding_2011,mazure_constructing_2018,schumaker_spline_2007}.
We begin by defining extended Tchebycheff spaces. Let $I$ be an interval of the real line.

\begin{definition}[Extended Tchebycheff spaces]
A $(p+1)$-dimensional subspace $\Uspace(I)$ of $C^{p}(I)$ is an  extended Tchebycheff space (\emph{ET-space}) on $I$ if 
any Hermite interpolation problem with $p+1$ data points in $I$, counting multiplicities, is unisolvent.
\end{definition}
\begin{remark}
	This definition implies that any non-trivial element of a $(p+1)$-dimensional ET-space admits at most $p$ zeros, counting multiplicities on $I$. Note that an ET-space on some interval $I$ is also an ET-space on any subinterval $J\subseteq I$.
\end{remark}

We identify the following relevant subclass of ET-spaces.

\begin{definition}[Extended complete Tchebycheff spaces]
    A $(p+1)$-dimensional subspace $\Uspace(I)$ of $C^{p}(I)$ is an  extended complete Tchebycheff space (\emph{ECT-space}) if there exists a basis  $\{u_{j}\}_{j=0}^{p}$ 
     of $\Uspace(I)$ such that, for all $k = 0, \dots, p$, 
    the subspace $\Span\{u_0, \dots, u_k\}$ is an ET-space on $I$.
     The basis 
    $\{u_{j}\}_{j=0}^{p}$ is called an ECT-system.
\end{definition}

An interesting approach to construct ECT-systems, and hence ECT-spaces, is based on a sequence $\{\omega_0, \dots, \omega_p\}$ of positive and smooth functions on $I$, called \emph{weight functions}, by means of repeated integration \cite{lyche_tchebycheffian_2019,mazure_finding_2011}.

\begin{definition}[Generalized power functions]
    Let $\{\omega_0, \dots, \omega_p\}$ be a sequence of weight functions on an interval $I$ satisfying
    \[
        \omega_j(x) > 0 \quad \text{and} \quad \omega_j \in C^{p-j}(I) \quad \text{for all } x \in I, \quad j=0,\dots,p.
    \]
    We define the generalized power functions $\{u_j\}_{j=0}^{p}$ by repeated integration as follows:
    \[
        \begin{aligned}
            u_0(x) &\coloneqq \omega_0(x),\\[4pt]
            u_1(x) &\coloneqq \omega_0(x)\int_{x_0}^x \omega_1(t_1)\dd t_1,\\[4pt]
            &\;\;\vdots\\[4pt]
            u_p(x) &\coloneqq \omega_0(x)\int_{x_0}^x \omega_1(t_1)
            \int_{x_0}^{t_1} \omega_2(t_2)
            \cdots
            \int_{x_0}^{t_{p-1}} \omega_p(t_p)\dd t_p\cdots \dd t_1,
        \end{aligned}
    \]
    where $x_0\in I$.
\end{definition}

It turns out that $\{u_0, \dots, u_p\}$ is an ECT-system on $I$.
The sequence $\{\omega_0, \dots, \omega_p\}$ is called a weight system generating the space $\Uspace(I) = \Span \{u_0,\dots,u_p\}$. ECT-systems are completely characterized by  weight systems.

\begin{theorem}[\cite{lyche_tchebycheffian_2019,schumaker_spline_2007}]
    The following statements are equivalent.
    \begin{enumerate}[label=(\roman*)]
        \item $\{u_0,\dots,u_p\}$ is an ECT-system on $I$.
        \item There exists a weight system generating the ECT-space $\Uspace(I) =\Span \{u_0,\dots,u_p\}$.
    \end{enumerate}
\end{theorem}

\begin{remark}
    A weight system generating an ECT-space is not unique. For example, for any positive constants $K_0,\dots,K_p$, the weight system $\{K_0 \omega_0,\dots,K_p \omega_p\}$ generates the same ECT-space as $\{\omega_0,\dots,\omega_p\}$.
\end{remark}

\begin{example}
    The weight system $\{\omega_j = 1\}_{j=0}^p$ generates the space of algebraic polynomials $\mathbb{P}_p$, which is an ECT-space on $\mathbb{R}$. For any fixed $x_0 \in \mathbb{R}$, the ECT-system obtained by repeated integration is given by
    \[
        \Bigl\{1,x-x_0,\dots,\frac{(x-x_0)^p}{p!}\Bigr\}.
    \]
     The obtained ECT-system agrees with the classical polynomial Taylor basis.
\end{example}

\begin{example}
	\label{ex:exp}
    Given distinct real numbers $\alpha_0<\dots<\alpha_p$, the weight functions 
    $$\omega_0 = e^{\alpha_0x}, \quad \omega_j = e^{(\alpha_j-\alpha_{j-1})x}, \quad j=1,\dots,p,$$
     generate the space
    \[
        \Span\bigl\{e^{\alpha_0x},\dots,e^{\alpha_px}\bigr\}.
    \]
    Hence, this is an ECT-space on any interval $I$ of the real line, in particular, for $I=\R$.
\end{example}

\begin{example}
	\label{ex:trig}
    For $\alpha> 0$ and $I = (-\frac{\pi}{2\alpha},\frac{\pi}{2\alpha})$, the weight functions
    $$\omega_0 = \cos(\alpha x), \quad \omega_1 = \frac{1}{\cos^2(\alpha x)},$$
    generate the space
    \[
        \Span\{\cos(\alpha x),\sin(\alpha x)\}.
    \]
    Hence, this is an ECT-space on $I = (-\frac{\pi}{2\alpha},\frac{\pi}{2\alpha})$.
\end{example}

We discuss next another construction of ECT-spaces based on the kernel of linear differential operators \cite{coppel_disconjugacy_1971,lyche_tchebycheffian_2019,schumaker_spline_2007}. The ECT-spaces in the examples above belong to this class.
Consider a linear differential operator $L$ defined on an interval $I$,
\[
    L \coloneqq D^{p+1} + \sum_{j=0}^{p}a_j(x)D^j, \quad \text{with} \quad a_{j}(x)\in C(I), \quad j=0,\dots,p+1.
\]
It is a well-known fact that $\Ker L \subseteq C^{p+1}(I)$
is a $(p+1)$–dimensional space.
We now connect $\Ker L$ with ET-spaces.

\begin{definition}
    The operator $L$ is said to be \emph{disconjugate} on $I$ if every non–trivial solution of $Lu=0$ has at most $p$ zeros on $I$, counting multiplicities.
\end{definition}

If $L$ is disconjugate on $I$, then its kernel is an ET-space on $I$. If the kernel of a linear differential operator $L$ is an ECT-space on $I$, then $L$ is disconjugate on $I$. However, the converse is not always true.

\begin{remark}
    It is known that for any disconjugate operator $L$ defined on a compact interval $I$, $\Ker L$ is an ECT-space on $I$. This is because the class of ECT-spaces on $I$ coincides with the larger class of ET-spaces on compact intervals \cite{coppel_disconjugacy_1971,lyche_tchebycheffian_2019}.
\end{remark}

The following result provides a characterization of linear differential operators whose kernels are ECT-spaces.

\begin{theorem}[\cite{coppel_disconjugacy_1971}]
    The kernel of a linear differential operator $L$ is an ECT-space on $I$ if and only if the operator can be represented as
    \[
        L = \omega_0 \omega_1 \cdots\omega_p \, D_p D_{p-1}\cdots D_0,
    \]
    where $\omega_0,\dots,\omega_p$ are positive functions on $I$, with
    \[
        \omega_j \in C^{p-j+1}(I), \quad \text{and} \quad D_jf \coloneqq D\left(\frac{f}{\omega_j}\right),
    \]
    for all $j=0,\dots,p$.
    Moreover, the functions $\omega_0,\dots,\omega_p$ form a  weight system generating the ECT-space $\Ker L$.
\end{theorem}

The next examples make use of this theorem to construct ECT-spaces from linear differential operators with constant coefficients. Consider a linear differential operator with constant coefficients,
\[
    L \coloneqq D^{p+1} + \sum_{k=0}^{p}a_j D^j,
\]
and its characteristic polynomial,
\begin{equation}
	\label{eq:pol-car}
    p_L(\xi) \coloneqq \xi^{p+1} + \sum_{k=0}^{p}a_j \ \xi^j.
\end{equation}

\begin{example}
    If $p_L$ has distinct real roots, namely $\alpha_0 < \dots < \alpha_p$,    then $L$ admits the following factorization on any real interval $I$:
    \[
        L =(D-\alpha_0)\cdots(D-\alpha_p) =e^{\alpha_px} D_p\cdots D_0,
    \]
    where $D_jf=D(f\,e^{-(\alpha_j-\alpha_{j-1})x})$ for all $j=1,\dots,p$ and $D_0f=D(f\,e^{-\alpha_0x})$.
    Hence,
    \[
        \Ker L = \Span\{e^{\alpha_0x},\dots,e^{\alpha_px}\}
    \]
    is an ECT-space on any real interval $I$; see also Example~\ref{ex:exp}.
\end{example}

\begin{example}
    Suppose $L = D^2 + \alpha^2$, with $\alpha > 0$. Then,
    \[
        L = \frac{1}{\cos(\alpha x)}D\left(\cos^2(\alpha x)D\left(\frac{f}{\cos(\alpha x)}\right)\right), \quad x \in \left(-\frac{\pi}{2\alpha},\frac{\pi}{2\alpha}\right).
    \]
    It is well known that
    \[
        \Ker(D^2+\alpha^2) = \Span\{\cos(\alpha x), \sin(\alpha x)\},
    \]
    corresponding to the roots $\pm \im \alpha$ of \eqref{eq:pol-car}. Hence, this space is an ECT-space on $(-\frac{\pi}{2\alpha},\frac{\pi}{2\alpha})$; see also Example~\ref{ex:trig}. Actually, it has been shown that this is an ECT-space on any interval with length less than $\frac{\pi}{\alpha}$ \cite{beccari_critical_2020,carnicer_critical_2003}.
\end{example}

\subsection{Tchebycheffian spline space on a multi-cell}

We begin by recalling a classical existence result of a global basis for the Tchebycheffian spline space in the univariate setting using the notation defined in Section~\ref{sec:generalized spline spaces}. We assume that each local space $\Uspace_i$ is an ECT-space generated by a  weight system $\boldsymbol{\omega}_i \coloneqq \{\omega_{i,0}, \dots, \omega_{i,p}\}$ defined on $[t_i,t_{i+1}]$ for $i=0,\dots ,n$. To ensure that these local spaces can be glued together to form a Tchebycheffian spline space $\Sspace(\{\Uspace_i\}_{i=0}^{n}, T, \mathbf{r})$ with a global B-spline-like basis, we impose compatibility conditions on the weights across the interior points.

\begin{definition}[Compatible weight system]
    The collection of local weight systems $\{ \boldsymbol{\omega}_i \}_{i=0}^{n}$ is said to be compatible with respect to the interior point sequence $T \coloneqq \{t_i\}_{i=1}^{n}$ and regularity vector $\mathbf{r} \coloneqq (r_1,\dots,r_n)$ if the following regularity conditions hold at every interior point $t_i$, $i=1,\dots,n$:
    \[
        D^k \omega_{i-1,j}(t_i) = D^k \omega_{i,j}(t_i), \quad \text{for } k=0, \dots, \max(r_i-j ,\, -1), \quad j=0,\dots,p.
    \]
\end{definition}

Under this compatibility assumption, we can define a system of \emph{global weight functions} $\boldsymbol{\omega} \coloneqq (\omega_0, \dots, \omega_p)$ on the entire domain $I$ by concatenating the local weights:
\[
    \omega_j(x) \coloneqq \omega_{i,j}(x), \quad \text{for } x \in [t_i, t_{i+1}).
\]

\begin{theorem}[\cite{buchwald_bsplines_2003,lyche_tchebycheffian_2019}]\label{thm:ect_assumption_basis}
    Given a compatible system of local weights, there exists a Tchebycheffian B-spline-like basis for the corresponding univariate spline space $\Sspace(\{\Uspace_i\}_{i=0}^{n}, T, \mathbf{r})$. This basis satisfies Assumption~\ref{ass:basis}.
\end{theorem}

Consequently, a global basis $\Bbasis_{\ect}$ for the tensor-product Tchebycheffian spline space is obtained as the tensor product of the univariate Tchebycheffian B-spline-like bases from the theorem above. This construction ensures the existence of a global basis with the required properties of Assumption~\ref{ass:basis} for any Tchebycheffian spline space derived from compatible local ECT-systems.
Let $\Ggrid$ be the multi-dimensional tensor grid associated with the tensor-product Tchebycheffian spline space. The local function space on a cell $I_{\mathbf{i}}$ is the tensor product of the univariate local ECT-spaces:
\[
    \Uspace_{\mathbf{i}} \coloneqq \bigotimes_{j=1}^{d} \Uspace_{i_j}^j.
\]
In the rest of this section, we denote the Tchebycheffian spline space by $\Sspace_{\ect}$ and the Tchebycheffian spline space on a multi-cell $M$ by $\Sspace_{\ect}(M)$.

With the global basis established, we verify the second requirement: Assumption~\ref{ass:extension}.

\begin{theorem} \label{thm:ect_assumption_verification}
    Assumption~\ref{ass:extension} is satisfied whenever the local spaces $\Uspace_c$ for $c\in \Ccell$ are ECT-spaces.
\end{theorem}

\begin{proof}
    We construct the extension inductively, relying on the tensor-product structure of the space and the unisolvence of Hermite interpolation in univariate ECT-spaces.

    \medskip
    \noindent \textbf{Reduction to the univariate setting:}
    Since the local space $\Uspace_c$ is a tensor product of univariate spaces $\Uspace^j$, it is spanned by basis functions of the form
    \[
        u_c(\mathbf{x}) \coloneqq \prod_{j=1}^d u_c^j(x_j), \quad  u_c^j \in \Uspace^j.
    \]
    By linearity, it suffices to construct extensions for elements in this form.
    The contact condition involves matching mixed partial derivatives $D^{\mathbf{k}} u_c = D^{\mathbf{k}} u_{c'}$ for all multi-indices $\mathbf{k} \coloneqq (k_1,\dots,k_d)$ prescribed by regularity conditions. Due to the product structure, the derivatives factorize as
    \[
        D^{\mathbf{k}} u_c(\mathbf{x}) = \partial_{x_1}^{k_1} u_c^1(x_1)\, \cdots\, \partial_{x_d}^{k_d} u_c^d(x_d).
    \]
    Consequently, the multivariate extension problem decouples into independent univariate extension problems in each direction. 

    \medskip
    \noindent \textbf{Univariate extension construction:}
    Consider adjacent cells corresponding to the intervals $I_i=[t_{i}, t_{i+1}]$ and $I_{i+1}=[t_{i+1}, t_{i+2}]$. To construct an extension $u_{\mathrm{ext}}$ to $I_{i+2}$:
    \begin{enumerate}
        \item \textbf{Fixed constraints:} The continuity at $t_{i+2}$ fixes the Hermite data (derivatives up to order $r_{i+2}$) at the left endpoint of $I_{i+2}$.
        \item \textbf{Free degrees of freedom}: We assign arbitrary values (e.g., zeros) to the remaining degrees of freedom at the right endpoint $t_{i+2}$ to match the dimension $p+1$.
        \item \textbf{Unisolvence:} Since $\Uspace_{i+2}$ is an ECT-space, the Hermite interpolation problem is unisolvent. Thus, a local section $u_{\mathrm{ext}} \in \Uspace_{i+2}$ exists.
    \end{enumerate}
    Repeating this process defines the global spline $s$.
\end{proof}

\begin{remark}
    ECT-spaces represent a particularly well-behaved class of spaces in the sense of Assumption~\ref{ass:extension}: they satisfy the extension property regardless of the regularity constraints imposed on the spline space.
\end{remark}

The Tchebycheffian spline space satisfies Assumptions~\ref{ass:basis} and \ref{ass:extension}. Thus, the general framework in Section~\ref{ssec:space multi-cell} applies directly.

\begin{corollary}
\label{cor:D(M) ECT}
    The set
    \[
        \Dbasis_{\ect}(\Mcell) \coloneqq \Bigl\{ \beta_{\Phi} \coloneqq (\beta_{\Phi,c})_{c \in \Mcell} \ \Big|\ \beta \in \Bbasis_{\ect}, \ \supp(\beta) \cap  \overset{\circ}\Mcell \neq \emptyset,\ \Phi \in K(\Ggraph_{\beta}) \Bigr\}
    \]
    forms a basis for $\Sspace_{\ect}(\Mcell)$. Under Assumption~\ref{ass:connectivity}, this set simplifies to
    \[
        \Bigl\{ \beta|_{\Mcell} = (\beta_c)_{c \in \Mcell} \ \Big|\ \beta \in \Bbasis_{\ect}, \ \supp(\beta) \cap  \overset{\circ}\Mcell \neq \emptyset \Bigr\}.
    \]
\end{corollary}

\subsection{Tchebycheffian hierarchical spline space}

In this section, we apply the general construction of the hierarchical basis to the Tchebycheffian setting. Let $\{\Sspace_{\ect,l}\}_{l=0}^{N-1}$ be a sequence of nested tensor-product Tchebycheffian spline spaces with the associated grids $\{\Ggrid_l\}_{l=0}^{N-1}$.
As established, for each level $l$, the space $\Sspace_{\ect,l}$ admits a tensor-product basis $\Bbasis_{\ect,l}$ composed of Tchebycheffian B-spline-like functions constructed from compatible weight systems. The framework developed above for generalized hierarchical spline spaces naturally includes the Tchebycheffian setting as a special case. Using the setting presented in Section~\ref{ssec:hierarchical space}, we introduce the space:
\[
    \Hspace_{\ect}(\Gamma)
    \coloneqq
    \Bigl\{
    s \coloneqq (s_l)_{l=0}^{N-1}
    \ \Big|\
    s_l\in \Sspace_{\ect,l}(\Gamma_l),
    \
    s_l \sim_{\h} s_k
    \ \forall\,l<k
    \Bigr\},
\]
which we call the \emph{Tchebycheffian hierarchical spline space}. By construction, $\Hspace_{\ect}(\Gamma)$ is a particular instance of the generalized hierarchical spline space $\Hspace(\Gamma)$ obtained when the local function spaces are tensor products of ECT-spaces.

Since Assumptions~\ref{ass:basis} and \ref{ass:extension} are satisfied in the Tchebycheffian context, we can simplify Assumption~\ref{ass:basis hierachichal setting} as follows.

\begin{assumption}[Hierarchy of Tchebycheffian spline spaces] \label{ass:basis Tchebycheffian hierachichal setting}
    The sequence $\{\Sspace_{\ect,l}\}_{l=0}^{N-1}$ satisfies the following:
    \begin{enumerate}
        \item the sequence is defined on a common domain $\Omega \supseteq \Gamma$;
        \item the sequence is nested ($\Sspace_{\ect,0}\subseteq \dots \subseteq \Sspace_{\ect,N-1}$).
   \end{enumerate}
\end{assumption}

\begin{corollary} \label{cor:Tchebycheffian hierachichal complete}
    The Tchebycheffian function set $\Hbasis_{\ect,\Gamma}$ is defined by applying the selection mechanism from Definition~\ref{def:selection} to the Tchebycheffian bases $\{\Bbasis_{\ect,l}\}_{l=0}^{N-1}$ as in \eqref{eq:hierarchical-basis}.
    Under Assumption~\ref{ass:basis Tchebycheffian hierachichal setting},  we have $\Span\Hbasis_{\ect,\Gamma} \subseteq \Hspace_{\ect}(\Gamma)$ and $\Hbasis_{\ect,\Gamma}$ is linearly independent.
    Additionally, if Assumption~\ref{ass:connectivity} is satisfied for each level, the set $\Hbasis_{\ect,\Gamma}$ is complete, i.e., $\Span\Hbasis_{\ect,\Gamma} = \Hspace_{\ect}(\Gamma)$.
\end{corollary}

\begin{proof}
    The result is a direct consequence of Theorems~\ref{thm:linind} and~\ref{thm:hier_complete} because the Tchebycheffian B-spline basis satisfies Assumption~\ref{ass:basis} (by Theorem~\ref{thm:ect_assumption_basis}) and the local tensor-product spaces satisfy Assumption~\ref{ass:extension} (by Theorem~\ref{thm:ect_assumption_verification}).
\end{proof}

\begin{remark}
If Assumption~\ref{ass:connectivity} fails, the setting introduced above allows us to construct verbatim the Tchebycheffian decoupled hierarchical basis $\Hbasis_{\ect,\Gamma}^{\dec}$ using the bases $\Dbasis_{\ect}(\Gamma_l)$ as in Section \ref{ssec:decoupled basis}. If the decoupling is feasible for each level $l$ (Definition~\ref{def:feas of decoup}), the completeness of the Tchebycheffian decoupled hierarchical basis is a direct consequence of Theorem~\ref{thm:decoupling complete}.
\end{remark}

\section{Conclusion} \label{sec:conclusion}

In this work, we established a unified theoretical framework for generalized hierarchical spline spaces. Our primary contribution is a shift in perspective from the classical \emph{constructive} approach, which relies on a specific basis selection mechanism, to a \emph{descriptive} formulation based on hierarchical contact conditions. This approach allows us to define the spline space as an intrinsic mathematical object, independent of its basis. Within this framework, we introduced the \emph{extension assumption} (Assumption~\ref{ass:extension}) as the key sufficient condition to link these two perspectives. We also identified rules under which the standard hierarchical basis selection is complete, spanning exactly the space defined by the regularity constraints.

To demonstrate the versatility of this theory, we applied it to the class of Tchebycheffian spline spaces. We showed that spaces generated by extended complete Tchebycheff (ECT) systems --- which are of interest in several applications, including isogeometric analysis --- satisfy the abstract requirements of our framework. By verifying the necessary compatibility and extension properties, we established the completeness of hierarchical Tchebycheffian splines, confirming that our generalized theory effectively covers this important non-polynomial setting.

Several open questions and directions for future research emerge from this study.
\begin{enumerate}
    \item \textbf{Homological dimension analysis:} The descriptive formulation introduced here, which characterizes splines via algebraic constraints across interfaces, is particularly well-suited for a homological treatment. 
    This approach promises to provide dimension formulas for complex configurations, where traditional combinatorial arguments become intractable; see \cite{mourrain_dimension_2014} and \cite{bracco_generalized_2016,bracco_dimension_2016,bracco_tchebycheffian_2019}.
    
    Future work could explore the use of homological algebra to study the dimension of hierarchical \emph{polynomial}  spline spaces.
    As a first step, one could develop a double complex encoding both cell-level contact conditions and hierarchical contact conditions in the polynomial setting, where the ring structure provides the necessary algebraic tools.
    
    The homological approach has been extended to Tchebycheffian spline spaces with uniform ECT-sections in \cite{bracco_dimension_2016,bracco_tchebycheffian_2019}, replacing the missing ring structure by  the so-called $\vd$-sum property for the (single) ECT-space.
    Extending the homological approach to (hierarchical) spline spaces whose sections belong to different ECT-spaces --- precisely the setting of the present work --- requires an appropriate generalization of this property.
    	
    \item \textbf{Truncated hierarchical B-splines (THB-splines):} Extending this generalized framework to THB-splines is a natural next step. Adapting the truncation mechanism to generalized spaces would combine the theoretical rigor of our completeness results with the superior stability and partition of unity properties of the truncated hierarchical basis relevant for several applications like isogeometric analysis; see \cite{giannelli_thb_2012,giannelli_strongly_2014}.

    \item \textbf{More general splines:} The term \emph{generalized splines} has been often used in the literature, with different meanings. Sometimes with a more restrictive meaning with respect to the definition adopted in the current work; see, e.g., \cite{manni_survey_2017}. There also exist definitions of generalized splines resulting in larger classes of functions. Among the others,  in \cite[Chapter~11]{schumaker_spline_2007}, a further generalization is obtained by considering regularity conditions identified by general differential operators, instead of classical derivatives. Extending our results to this more general framework seems to be affordable.
\end{enumerate}

\section*{Acknowledgements}
C.~Manni and H.~Speleers are members of the research group GNCS (Gruppo Nazionale per il Calcolo Scientifico) of INdAM (Istituto Nazionale di Alta Matematica).
They acknowledge support from the MUR Excellence Department Project MatMod@TOV (CUP E83C23000330006) awarded to the Department of Mathematics of the University of Rome Tor Vergata.
They are also partially supported by
a Project of Relevant National Interest (PRIN) under the National Recovery and Resilience Plan (PNRR) funded by the European Union --- Next Generation EU (CUP E53D23017910001) and
by the Italian Research Center in High Performance Computing, Big Data and Quantum Computing (CUP E83C22003230001).

\bibliographystyle{plain}
\bibliography{completeness}

\end{document}